\pdfoutput=1
\RequirePackage{ifpdf}
\ifpdf % We~are running pdfTeX in pdf mode
\documentclass[pdftex]{sigma}
\else
\documentclass{sigma}
\fi

\numberwithin{equation}{section}

\newtheorem{Theorem}{Theorem}[section]
\newtheorem{Corollary}[Theorem]{Corollary}
\newtheorem{Lemma}[Theorem]{Lemma}
\newtheorem{Proposition}[Theorem]{Proposition}
\newtheorem{Conjecture}[Theorem]{Conjecture}

{\theoremstyle{definition}
\newtheorem{Definition}[Theorem]{Definition}
\newtheorem{Example}[Theorem]{Example}
\newtheorem{Remark}[Theorem]{Remark} }

\begin{document}

\allowdisplaybreaks

\renewcommand{\thefootnote}{}

\newcommand{\arXivNumber}{2402.17174}

\renewcommand{\PaperNumber}{063}

\FirstPageHeading

\ShortArticleName{On Some Special Cases of Gaiotto's Positivity Conjecture}

\ArticleName{On Some Special Cases\\ of Gaiotto's Positivity Conjecture\footnote{This paper is a~contribution to the Special Issue on Basic Hypergeometric Series Associated with Root Systems and Applications in honor of Stephen C.~Milne's 75th birthday. The~full collection is available at \href{https://www.emis.de/journals/SIGMA/Milne.html}{https://www.emis.de/journals/SIGMA/Milne.html}}}

\Author{Pavel ETINGOF}

\AuthorNameForHeading{P.~Etingof}

\Address{Department of Mathematics, MIT, Cambridge, MA 02139, USA}
\Email{\href{mailto:etingof@math.mit.edu}{etingof@math.mit.edu}}
\URLaddress{\url{http://www-math.mit.edu/~etingof}}

\ArticleDates{Received February 29, 2024, in final form July 11, 2024; Published online July 13, 2024}

\Abstract{We prove a conjecture of D.~Gaiotto on positivity of inner products arising in studying Landau--Ginzburg boundary conditions in the 1-dimensional case, and in special cases in higher dimensions, for 3d free hypermultiplets.}


\Classification{81T13; 42A82; 15B48}

\renewcommand{\thefootnote}{\arabic{footnote}}
\setcounter{footnote}{0}

\section{Introduction}

Suppose that $V=\bigcup_{N\ge 0}F_NV$ is a complex vector space with an ascending filtration, such that ${\dim F_NV<\infty}$ for all $N$. The associated graded space to $V$ is
\[
\operatorname{gr}V:=\bigoplus_{N\ge 0}\operatorname{gr}_NV,
\]
where $\operatorname{gr}_NV:=F_NV/F_{N-1}V$. If $V$, $V'$ are two such filtered vector spaces, then a linear map~${\phi\colon V'\to V}$
is filtered if $\phi(F_NV')\subset F_NV$. In this case, $\phi$ gives rise to the associated graded map
$\operatorname{gr}\phi\colon \operatorname{gr}V'\to \operatorname{gr}V$. Note also that $\operatorname{gr}V$ may be viewed as a filtered space
with~${F_N\operatorname{gr}V:=\oplus_{j=0}^N \operatorname{gr}_jV}$.

Let $(\,,\,)$ be a bilinear or sesquilinear form on a filtered space $V$. Let us say that $(\,,\,)$ is \emph{filtered-nondegenerate} if its restriction $(\,,\,)_N$ to the subspace $F_NV$ is nondegenerate for every $N\ge 0$. In this case we have
unique filtered isomorphisms $\phi_+,\phi_-\colon \operatorname{gr}V\to V$
with $\operatorname{gr}\phi_+=\operatorname{gr}\phi_-={\rm Id}$
such that for all $N$, $\phi_\pm(\operatorname{gr}_{N}V)=(F_{N-1}V)^\perp_\pm$,
where for $Y\subset F_NV$, $Y^\perp_+$, $Y^\perp_-$
denote the left, respectively right, orthogonal complement of~$Y$
under $(\,,\,)$ (the \emph{Gram--Schmidt biorthogonalization maps}). Thus we get a bilinear or sesquilinear form
$\langle\,,\,\rangle$ on $\operatorname{gr}V$ given by
$
\langle v_1,v_2\rangle:=(\phi_+v_1,\phi_-v_2)$.
Moreover, this form is nondegenerate when restricted
to each homogeneous component $\operatorname{gr}_NV$ of $\operatorname{gr}V$, and different
homogeneous components are orthogonal.

We will consider this construction when $V=\mathbb C[z_1,\dots,z_n]$
and the filtration is given by degree of polynomials, so that
$F_NV=\mathbb C[z_1,\dots,z_n]_{\le N}$, the space of polynomials
of degree at most $N$. Namely, let $W\in \mathbb C[z_1,\dots,z_n]$ be a homogeneous polynomial of degree $d$ (in physical terminology, a Landau--Ginzburg superpotential). Consider the sesquilinear form on $\mathbb C[z_1,\dots,z_n]$ defined by the formula{\samepage
\[
(P,Q)_W=\int_{\mathbb C^n}P(z)\overline {Q(z)}{\rm e}^{-|z|^2+W(z)-\overline{W(z)}}{\rm d}z{\rm d}\overline z.
\]
It is evident that this form is $\mathbb Z/d$-invariant, i.e., $(P,Q)_W=(P_*,Q_*)_W$, where
$P_*(z):=P\big({\rm e}^{\frac{2\pi {\rm i}}{d}} z\big)$.}

It is clear that the form $(\,,\,)_0$ is filtered-nondegenerate, as the monomial basis is orthogonal with respect to $(\,,\,)_0$, with positive norms. Hence $(\,,\,)_W$ is filtered-nondegenerate for sufficiently generic $W$ (outside the zero sets of a countable collection of polynomials on $S^d\mathbb C^{n*}$ regarded as a real vector space).
Moreover, while $(\,,\,)_W$ is not, in general, Hermitian-symmetric, we have
\[
(P,Q)_W=\overline {(Q_*,P_*)_W}.
\]
Hence if for some $W$ the form $(\,,\,)_W$ is filtered-nondegenerate then the resulting homogeneous form $\langle\,,\,\rangle_W$ on $\mathbb C[z_1,\dots,z_n]$ is Hermitian symmetric.

Recently D.\ Gaiotto, motivated by the study of Landau--Ginzburg boundary conditions for 3d free hypermultiplets, proposed the following conjecture (in the more general setting of quasi-homogeneous $W$; see \cite[Sections 7.6 and 7.7]{G}).

\begin{Conjecture}\label{c1}\quad
\begin{itemize}\itemsep=0pt
\item[$(i)$] If the form $(\,,\,)_W$ is filtered-nondegenerate, then the form $\langle\,,\,\rangle_W$ is positive definite.
\item[$(ii)$] The form $(\,,\,)_W$ is filtered-nondegenerate for all $W\in S^d\mathbb C^{n*}$.
\end{itemize}
\end{Conjecture}

Note that the form $\langle\,,\,\rangle_0$ is clearly positive definite.
Thus part~(i) of Conjecture~\ref{c1} follows from part~(ii) by a deformation argument.
In fact, part~(i) follows even from a weaker version of part (ii), stating that $(\,,\,)_W$ is filtered-nondegenerate on a dense
connected subset of $S^d\mathbb C^{n*}$. Moreover, we already know
that this holds on a dense subset, but it is not clear that this subset is connected, since we removed zero sets of {\it real} polynomials.

In particular, for $N=0$ Conjecture \ref{c1}\,(ii) reduces to

\begin{Conjecture}\label{c2} The real number
\[
I(W):=\int_{\mathbb C^n}{\rm e}^{-|z|^2+W(z)-\overline{W(z)}}{\rm d}z{\rm d}\overline z
\]
is strictly positive for all $W$ $($or, equivalently, nonzero for all $W)$.
\end{Conjecture}

These conjectures are easy to check directly for $d\le 2$ using Gaussian integrals, but already the case $d=3$ is not known.

We prove two results in the direction of these conjectures.

\begin{Theorem}\label{t1} Conjecture {\rm\ref{c1}} holds for
$n=1$.
\end{Theorem}

Theorem \ref{t1} is proved in Section \ref{sec2} using (the ``easy direction'' of) Schoenberg's theorem on totally positive functions. The relevant inner product in this case is defined by \cite[formula~(7.34)]{Ga}.

\begin{Theorem}\label{t2} Conjecture {\rm\ref{c2}} holds if $n\le d$.
\end{Theorem}

Theorem \ref{t2} is proved in Section \ref{sec3}; it easily follows from Bochner's theorem on positive definite functions.

\begin{Remark}\quad
\begin{itemize}\itemsep=0pt
\item[(1)] It is clear that if Conjecture \ref{c1} holds for
$W_1$ and $W_2$ then it holds for $(W_1\boxtimes W_2)(z,w):=W_1(z)W_2(w)$. Thus Theorem \ref{t1} implies
that Conjecture \ref{c1} holds for
\[
W(z_1,\dots,z_n)=a_1z_1^d+\dots+a_nz_n^d.
\]
More generally, by the same argument it holds for
\[
W(z_1,\dots,z_n)=a_1z_1^{d_1}+\dots+a_nz_n^{d_n}
\]
in the quasi-homogeneous setting.
\item[(2)] Let $t>0$. Making a change of variable $z\mapsto tz$ in the integral, we obtain
\[
(P,Q)_W=t^{2n}\int_{\mathbb C^n}P(tz)\overline {Q(tz)}{\rm e}^{-t^2|z|^2+t^{d}(W(z)-\overline{W(z)})}{\rm d}z{\rm d}\overline z.
\]
Thus setting $P_t(z):=P\big(t^{-1}z\big)$, we may define the inner product
\[
(P,Q)_{W,t}:=t^{-2n}(P_t,Q_t)_{t^{-d}W}=\int_{\mathbb C^n}P(z)\overline {Q(z)}{\rm e}^{-t^2|z|^2+W(z)-\overline{W(z)}}{\rm d}z{\rm d}\overline z.
\]
If $W$ defines an isolated singularity (i.e., the projective hypersurface $W=0$ is smooth), then we can take the limit $t\to 0$ and get the inner product
\[
(P,Q)_{W,0}=\int_{\mathbb C^n}P(z)\overline {Q(z)}{\rm e}^{W(z)-\overline{W(z)}}{\rm d}z{\rm d}\overline z
\]
(the integral is conditionally convergent). It is explained in \cite[Section 7.7 and Appendi\-ces~A--C]{G} that this inner product arises in 2d Landau--Ginzburg models, and Conjecture~\ref{c1} holds for it (namely, the positivity is shown by using the Morse flow to deform the contour of integration so that the integrand is manifestly positive). This implies that for any~$W$ defining an isolated singularity, Conjecture~\ref{c2} holds for~$\varepsilon^{-1}W$ for sufficiently small~$\varepsilon$.
\end{itemize}
\end{Remark}

\section[Proof of Theorem 1.3]{Proof of Theorem \ref{t1}}\label{sec2}

\subsection{Totally positive functions} We start with reviewing the theory
of totally positive functions which is used in the proof of Theorem \ref{t1}.
More details can be found in \cite{K2,K}.

\begin{Definition} A function $f\colon \mathbb R\to \mathbb R$
is \emph{totally positive} if for any real
\[
x_1<\dots<x_N,\qquad y_1<\dots<y_N,
\]
 we have
$
\det_{1\le i,j\le N} f(x_i-y_j)\ge 0$,
and is \emph{strictly totally positive} if
this inequality is strict ($\det_{1\le i,j\le N} f(x_i-y_j)>0$).
\end{Definition}

We will consider only continuous totally positive functions.

For a function $f\colon \mathbb R\to \mathbb R$, denote by $\widehat f$ its Fourier transform
\[
\widehat f(s)=\frac{1}{\sqrt{2\pi}}\int_{\mathbb R}{\rm e}^{{\rm i}sx}f(x){\rm d}x.
\]

\begin{Lemma}\label{l0} Let $f$ be a totally positive $L^1$-function. Suppose that for some $N$ and some $y_1<\dots<y_N$, the function $F_{\mathbf{y}}(x_1,\dots,x_N):=\det_{1\le i,j\le N} f(x_i-y_j)$ is identically zero. Then~${f=0}$.
\end{Lemma}

\begin{proof} Let $N$ be the smallest such number. Expanding the determinant, we obtain
\[
\sum_{k=1}^{N} C_{k}(\boldsymbol{y},x_1,\dots,x_{N-1})f(x_N-y_k)=0,
\]
where $C_{k}(\boldsymbol{y},x_1,\dots,x_{N-1}):=(-1)^{N-k}\det_{1\le i\le N-1,1\le j\ne k\le N} f(x_i-y_j)$. Moreover, by assumption there exist $x_1<\dots<x_{N-1}$ such that $C_{N}(\boldsymbol{y},x_1,\dots,x_{N-1})\ne 0$.
Passing to Fourier transforms with respect to $x_N$, we have
\[
\Bigg(\sum_{k=1}^N C_k(\boldsymbol{y},x_1,\dots,x_{N-1}){\rm e}^{{\rm i}s y_k}\Bigg)\widehat f(s)=0.
\]
Thus $\widehat f=0$,
hence $f=0$.
\end{proof}

Let $f\colon \mathbb R\to \mathbb R$ be a bounded measurable function and $g\colon \mathbb R\to \mathbb R$ be an $L^1$-function. Recall that the convolution $f*g$
is the function
\[
(f*g)(x)=\int_{\mathbb R}f(x-y)g(y){\rm d}y.
\]

\begin{Proposition}[\cite{K2}]\label{proo1}\quad
\begin{itemize}\itemsep=0pt
\item[$(i)$] If $f$, $g$ are totally positive, then so is $f*g$.
\item[$(ii)$] If moreover $f$ is strictly totally positive, then so is $f*g$.
\end{itemize}
\end{Proposition}

\begin{proof} (i) follows from the Cauchy--Binet formula \cite[Lemma 0.1]{K2}:
\[
\det_{1\le i,j\le N}((f*g)(x_i-z_j))=\int_{y_1<\dots<y_N}\det_{1\le i,k\le N}f(x_i-y_k)\det_{1\le k,j\le N}g(y_k-z_j){\rm d}\boldsymbol{y}.
\]

(ii) follows from the Cauchy--Binet formula and Lemma \ref{l0}.
\end{proof}

\begin{Example}\label{exaa} The determinant
$\det_{1\le i,j\le n}({\rm e}^{u_iv_j})$ for $u_1<\dots<u_N$ and $v_1<\dots<v_N$
is positive \cite[Section~XIII.8]{Ga}. Thus, setting $u_i:={\rm e}^{x_i}$, $v_i:=-{\rm e}^{-y_i}$, we get that for any real $x_1<\dots<x_N$, $y_1<\dots<y_N$ the determinant \smash{$\det_{1\le i,j\le N}\big({\rm e}^{-{\rm e}^{x_i-y_j}}\big)$} is positive.
Hence the function
$
f(x)={\rm e}^{-{\rm e}^{x}}
$
is strictly totally positive (and bounded).
\end{Example}

Thus Proposition \ref{proo1} implies

\begin{Proposition}\label{stp}
For every totally positive $L^1$-function $g$, the function ${\rm e}^{-{\rm e}^{x}}*g$
is strictly totally positive.
\end{Proposition}

For $f\in C^\infty(\mathbb R)$, consider the Wronskians
\[
\Delta_N(f,x):=\det_{0\le k,\ell\le N-1}\big((-1)^\ell f^{(k+\ell)}(x)\big).
\]

\begin{Definition}[\cite{K}] The function $f$ is said to be \emph{extended totally positive} if
$\Delta_N(f,x)$ is strictly positive for all $N$ and all $x\in \mathbb R$.
\end{Definition}

\begin{Proposition}[{\cite[p.\ 55]{K}; see also \cite[p.\ 41]{K2}}]\label{lee1}
 If $f\colon\mathbb R\to \mathbb R$
is analytic, then it is strictly totally positive if and only if
it is extended totally positive.
\end{Proposition}

\begin{Theorem}[Schoenberg \cite{S}; see \cite{Gr, K} for a review]\label{th2}
A non-zero $L^1$-function $f(x)$ is totally positive iff its Fourier transform $G(s)$ is of the form $1/\Psi({\rm i}s)$, where $\Psi$ is a \emph{Laguerre--Polya entire function} with $\Psi(0)>0$. This means that $G(s)$ admits a Hadamard factorization
\[
G(s)=C{\rm e}^{-\beta s^2+{\rm i}\alpha s}\prod_{j\in J} \left(1+\frac{{\rm i}s}{a_j}\right)^{-1}{\rm e}^{\frac{{\rm i}s}{a_j}},
\]
where $C>0$, $\beta\ge 0$, $a_j,\alpha\in \mathbb R$, $\sum_{j\in J} a_j^{-2}<\infty$ and $J\ne \varnothing$ if $\beta=0$.
\end{Theorem}

\begin{Example}\label{gammafu} The Fourier transform of the (strictly) totally positive function
${\rm e}^{-{\rm e}^{x}}$ is $\frac{1}{\sqrt{2\pi}}\Gamma({\rm i}s)$, which has the Hadamard factorization
\[
\Gamma({\rm i}s)=\tfrac{{\rm e}^{-\gamma {\rm i}s}}{{\rm i}s}\prod_{n=1}^\infty \left(1+\frac{{\rm i}s}{n}\right)^{-1}{\rm e}^{\frac{{\rm i}s}{n}},
\]
where $\gamma$ is the Euler constant. This product has a pole
at $s=0$ since ${\rm e}^{-{\rm e}^x}\notin L^1$. However, for any $\varepsilon>0$
the strictly totally positive function ${\rm e}^{\varepsilon x-{\rm e}^x}$ is in $L^1$ and has the Fourier transform~${\frac{1}{\sqrt{2\pi}}\Gamma({\rm i}s+\varepsilon)}$ with Hadamard factorization exactly in the form of Theorem \ref{th2}.
\end{Example}

\subsection{Oscillating integrals} \label{osint}

Fix an integer $d>2$. For an integer $0\le p\le d-1$ and $t\in \mathbb R$, let
\[
F_{p}(t):=\frac{1}{\pi d}\int_{\mathbb C}|z|^{2p}{\rm e}^{-|z|^2+2{\rm i}t\operatorname{Re}z^d}|{\rm d}z|^2.
\]
(This function, of course, depends on $d$, but we will not indicate
this dependence in the notation.)
Let us make a change of variable $w=z^d$. We then get
\[
F_p(t)=\frac{1}{\pi}\int_{\mathbb C}|w|^{\frac{2p}{d}}{\rm e}^{-|w|^{2/d}+2{\rm i}t\operatorname{Re}w}\big|w^{\frac{1-d}{d}}{\rm d}w\big|^2.
\]
Thus the function $F_p$ is the restriction to the real axis of the Fourier transform of the distribution
\[
\widehat F_p(w)={\rm e}^{-|w|^{2/d}}|w|^{\frac{2(p+1)}{d}-2}|{\rm d}w|^2=\sum_{k\ge 0}(-1)^k\frac{|w|^{\frac{2(k+p+1)}{d}-2}}{k!}|{\rm d}w|^2.
\]
Recall that the Fourier transform of
$|w|^s|{\rm d}w|^2$ is $\frac{\Gamma(\frac{s}{2}+1)}{\Gamma(-\frac{s}{2})}|t|^{-s-2}$.
Thus we have
\[%\label{asy}
F_p\big({\rm e}^{\frac{u}{2}}\big)=\sum_{k\ge 0}(-1)^k\frac{\Gamma\big(\frac{k+p+1}{d}\big)}{k!\Gamma\big(1-\frac{k+p+1}{d}\big)}{\rm e}^{-\frac{(k+p+1)u}{d}},\qquad u\in \mathbb R.
\]
Hence $F_p\big({\rm e}^{\frac{u}{2}}\big)=O\big({\rm e}^{-(p+1)\frac{u}{d}}\big)$ as $u\to +\infty$ (the terms with $k>0$ are dominated by
the term with $k=0$ in this limit).

\begin{Proposition}\label{stpos} For $p\le d-1$, the function $F_p\big({\rm e}^{\frac{u}{2}}\big)$ is strictly totally positive.
\end{Proposition}

\begin{proof}
Let $G_p(s)$ be the Fourier transform of $F_p\big({\rm e}^{\frac{u}{2}}\big)$ multiplied by $\sqrt{2\pi}$. Making a change of variable $v={\rm e}^{u/2}$, we get
\begin{align*}
G_p(s)&=\frac{1}{\pi d}\int_{-\infty}^\infty\int_{\mathbb C}|z|^{2p}\exp\bigl(-|z|^2+2{\rm i}{\rm e}^{\frac{u}{2}}\operatorname{Re}z^d+{\rm i}us\bigr)|{\rm d}z|^2{\rm d}u
\\
&=\frac{2}{\pi d}\int_{0}^\infty\int_{\mathbb C}|z|^{2p}\exp\bigl(-|z|^2+2{\rm i}v\operatorname{Re}z^d\bigr)v^{2{\rm i}s-1}|{\rm d}z|^2{\rm d}v.
\end{align*}
Let us split this integral into two parts, one for $\operatorname{Re}\big(z^d\big)\ge 0$
and one for $\operatorname{Re}\big(z^d\big)\le 0$, and replace integration from
$0$ to $\infty$ by integration from $0$ to $+{\rm i}\infty$ on the first part and to $-{\rm i}\infty$ on the second part.

This contour manipulation is legitimate for the following reason. Suppose $\operatorname{Re}\big(z^d\big)> 0$. The difference between the integrals in $v$ from $0$ to $R$ and from $0$ to ${\rm i}R$ is the integral over the quarter-circle from $R$ to ${\rm i}R$, which we can write as the integral from $0$ to $\pi/2$ with respect to $\theta={\rm arg}v$. But this integral goes to zero as $R\to \infty$ since for $\operatorname{Im}v>0$, $|v|=R$, we have~${\big|\exp \big(2{\rm i}v\operatorname{Re}\big(z^d\big)\big)\big|=\exp\bigl(-2R \sin \theta\cdot \operatorname{Re}\big(z^d\big)\bigr)}$, hence the integrand goes to zero as $R\to \infty$ for $0<\theta<\frac{\pi}{2}$. The case $\operatorname{Re}\big(z^d\big)< 0$ is similar.

The above change of contours creates factors ${\rm e}^{\pm \pi s}$, so we get
\begin{align*}
G_p(s)&=\frac{4}{\pi d}\cosh(\pi s)\Gamma(2{\rm i}s)\int_0^{\pi}|2\cos\theta|^{-2{\rm i}s}{\rm d}\theta\int_0^\infty {\rm e}^{-r^2}r^{2p+1-2{\rm i} ds}{\rm d}r
\\
&=\frac{4}{\sqrt{\pi} d}\frac{\Gamma ({\rm i}s)}{\Gamma(1-{\rm i}s)}\Gamma(p+1-{\rm i} ds).
\end{align*}
Now the crucial observation is that part of the poles of $\Gamma(p+1-{\rm i}ds)$ cancel all the zeros of~$\frac{1}{\Gamma(1-{\rm i}s)}$ to bring the function to the zero-free Laguerre--Polya form.
Namely, writing the Euler product, we get
$
G_p(s)=\Gamma({\rm i}s)H_p(s)$,
where
\[
H_p(s):=\frac{4}{\sqrt{\pi} d}{\rm e}^{{\rm i}\gamma (d-1)s}\prod_{m\ge p+1,\,m\notin d\mathbb Z} \left(1-\frac{{\rm i}ds}{m}\right)^{-1}{\rm e}^{-\frac{{\rm i}ds}{m}}.
\]
By Theorem \ref{th2}, we get
\[
F_p\big({\rm e}^{\frac{u}{2}}\big)=\int_{\mathbb R}{\rm e}^{-{\rm e}^{u-y}}h_p(y){\rm d}y,
\]
where $h_p$ is a totally positive $L^1$-function with Fourier transform $H_p(s)$.
Thus Proposition \ref{stp} implies that $F_p\big({\rm e}^{\frac{u}{2}}\big)$ is strictly totally positive, as claimed.\end{proof}

\begin{Remark} The bound $p\le d-1$ in Proposition \ref{stpos} is sharp:
we have
\[
F_d(t)\sim -\frac{\Gamma\big(1+\frac{1}{d}\big)}{\Gamma\big(1-\frac{1}{d}\big)d}t^{-2(1+\frac{1}{d})}
\]
as $t\to +\infty$, so it is negative for large $t$, even though $F_d(0)>0$.
\end{Remark}

\subsection{Proof of Theorem \ref{t1}}
Propositions \ref{stpos} and \ref{lee1} imply

\begin{Corollary}\label{c3} For every $N$ and $u\in \mathbb R$, we have
$
\det_{0\le k,\ell\le N-1}\big((-1)^\ell \partial_u^{k+\ell}F_p\big({\rm e}^{\frac{u}{2}}\big)\big)> 0$.
\end{Corollary}

Thus we obtain

\begin{Corollary}\label{c4} For every $N$ and $t\in \mathbb C^\times$, we have
$
\det_{0\le k,\ell\le N-1}\big((-1)^\ell \partial_t^k\overline \partial_t^\ell F_p(|t|)\big)> 0$.
\end{Corollary}

\begin{proof} It suffices to show that
$
\det_{0\le k,\ell\le N-1}\big((-1)^\ell t^k\partial_t^k\overline t^\ell\overline \partial_t^\ell F_p(|t|)\big)>0$.
But by row and column transformations, this determinant equals the determinant
\[
D_N:=\det_{0\le k,\ell\le N-1}\big((-1)^\ell (t\partial_t)^k(\overline t\overline \partial_t)^\ell F_p(|t|)\big).
\]
Now write $t={\rm e}^w$, then we get
\[
D_N=\det_{0\le k,\ell\le N-1}\big((-1)^\ell \partial_w^k\overline \partial_w^\ell F_p\big({\rm e}^{\frac{w+\overline w}{2}}\big)\big)=
\det_{0\le k,\ell\le N-1}\big((-1)^\ell \partial_u^{k+\ell} F_p\big({\rm e}^{\frac{u}{2}}\big)\big)|_{u=w+\overline w}.
\]
So $D_N> 0$ by Corollary \ref{c3}.
\end{proof}

Finally, note that
\[
\frac{1}{\pi d}\int_{\mathbb C}|z|^{2p}{\rm e}^{-|z|^2+tz^n-\overline t \overline z^n}|{\rm d}z|^2=F_p(|t|).
\]
Thus from Corollary \ref{c4}, we get the following result.

\begin{Proposition}\label{final} We have
$\det_{0\le k,\ell\le N-1}\big((-1)^\ell I_{k,\ell,p}(t)\big)>0$,
where
\[
I_{k,\ell,p}(t)=\int_{\mathbb C}z^{p+nk}\overline z^{p+n\ell}{\rm e}^{-|z|^2+tz^n-\overline t \overline z^n}|{\rm d}z|^2.
\]
\end{Proposition}

Proposition \ref{final} implies that if we define \emph{biorthogonal polynomials} $p_k(z)$, $q_k(z)$ using the weight function ${\rm e}^{-|z|^2+tz^n-\overline t \overline z^n}$
as explained in the introduction \big(i.e., by the formulas ${p_k\!:=\phi_+\big(z^k\big)}$, $q_k:=\phi_-\big(z^k\big)$\big), then
\[
\int_{\mathbb C}p_k(z)\overline{q_\ell(z)}{\rm e}^{-|z|^2+tz^n-\overline t \overline z^n}|{\rm d}z|^2=\delta_{kl}h_k,
\]
where $h_k>0$. This completes the proof of Theorem \ref{t1}.

\section[Proof of Theorem 1.4]{Proof of Theorem \ref{t2}}\label{sec3}
Recall that a function $\phi\colon \mathbb R\to \mathbb C$ such that $\phi(-x)=\overline{\phi(x)}$
is \emph{positive definite} if for distinct ${x_1,\dots,x_N\!\in\!\mathbb R}$ the matrix $(\phi(x_i-x_j))$ is positive definite; equivalently, for distinct ${x_1,\dots,x_N\!\in\!\mathbb R}$, $\det \phi(x_i-x_j)>0$.
\emph{Bochner's theorem} \cite{RS} states that a continuous function $\phi$ is positive definite iff it is the Fourier transform of a finite positive measure.

\begin{Example}\label{gammafu1} By Example~\ref{gammafu}, the function
\smash{$\frac{1}{\sqrt{2\pi}}\Gamma(\pm {\rm i}s+\varepsilon)$} for $\varepsilon>0$ is positive definite, since its Fourier transform is the positive $L^1$-function ${\rm e}^{\pm\varepsilon x-{\rm e}^{\pm x}}$.
\end{Example}

Let $n\le d$ and let $\rho\ne 0$ be a non-negative locally integrable function on $\mathbb C^n$ such that $\rho(\lambda z)=|\lambda|^{2\ell}\rho(z)$, $\lambda\in \mathbb C^*$ for some integer $\ell\le d-n$.
Consider the integral
\[
I(\rho,W):=\int_{\mathbb C^n}\rho(z){\rm e}^{-|z|^2+W(z)-\overline{W(z)}}{\rm d}z{\rm d}\overline z.
\]
Theorem \ref{t2} is the special case $\ell=0$, $\rho=1$ of the following theorem.

\begin{Theorem}\label{pro1} We have $I(\rho,W)>0$.
\end{Theorem}

\begin{proof}
For $z\in \mathbb C^n$ with $|z|=1$, consider the integral
\[
J(z):=\int_{\mathbb C}{\rm e}^{-|w|^2+W(z)w^d-\overline{W(z)w^d}}|w|^{2n+2\ell-2}{\rm d}w{\rm d}\overline w.
\]
We have
$
J(z)\!=\!F_{n+\ell-\!1}(|W\!(z)|)$.
Recall from Section \ref{osint} that the Fourier transform of $F_{n+\ell-\!1}\big({\rm e}^{\frac{u}{2}}\big)$
is proportional (with positive coefficient) to
\[
G(s):=\frac{\Gamma({\rm i}s)}{\Gamma(1-{\rm i}s)}\Gamma(n+\ell-{\rm i}ds).
\]
So by Example \ref{gammafu1}, $G(s)$ is positive definite for $n+\ell\le d$ (as it is then a product of Gamma functions, namely the Gamma function in the denominator cancels and the resulting function has no zeros).
Thus $F_{n+\ell-1}(t)\ge 0$ for $t>0$, i.e., $J(z)\ge 0$ (this was also shown using total positivity in Section \ref{sec2}).
Also this function is invariant under multiplication by ${\rm e}^{{\rm i}\theta}$, so descends to $\mathbb C\mathbb P^{n-1}$.
Now note that
\[
I(\rho,W)=
\int_{\mathbb C\mathbb P^{n-1}}\rho(z)J(z){\rm d}z{\rm d}\overline z.
\]
Thus $I(\rho,W)>0$ (as the function $J$ is real analytic, hence $\rho J$ is a non-negative function which is strictly positive on some subset of positive measure).
\end{proof}

Note that we in fact obtain the following generalization of Theorem \ref{t2}.

\begin{Corollary}
For any homogeneous polynomial $P\ne 0$ on $\mathbb C^n$ of degree $\ell$ with
$\ell+n\le d$,
we have
\[
I(W)(P):=\int_{\mathbb C^n}|P(z)|^2{\rm e}^{-|z|^2+W(z)-\overline{W(z)}}{\rm d}z{\rm d}\overline z > 0.
\]
\end{Corollary}

\begin{proof} We take $\rho(z)=|P(z)|^2$ in Proposition \ref{pro1}.
\end{proof}

In other words, Conjecture \ref{c1} holds up to degree $d-n$.

\begin{Remark} In the borderline case $\ell=d-n$ (when $d\ge n$), homogeneous polynomials of degree $\ell$ on $\mathbb C^n$ can be interpreted as holomorphic volume forms on the projective hypersurface $X\subset \mathbb C\mathbb P^{n-1}$ defined by the equation $W=0$. The integral $I(W)(P)$ is then the squared norm
\[
I(W)(P)=\|P\|^2= {\rm i}^{n-1}\int_X P\wedge \overline P,
\]
which is manifestly positive.
\end{Remark}

\subsection*{Acknowledgements}
I am very grateful to D.~Gaiotto for posing the problem and useful discussions. This work was partially supported by the NSF grant DMS-2001318.

\pdfbookmark[1]{References}{ref}
\LastPageEnding

\end{document}